\documentclass[a4paper, oneside]{amsart}

\usepackage{geometry}
\RequirePackage[l2tabu, orthodox]{nag}
\usepackage[all, warning]{onlyamsmath}

\usepackage{amsmath}
\usepackage{amsthm}
\usepackage{amsfonts}
\usepackage{amssymb}
\usepackage{amscd}

\usepackage{mathrsfs}
\usepackage{bm}
\usepackage{bbm}
\usepackage{stmaryrd}
\usepackage{braket}

\usepackage[all]{xy}
\usepackage{pb-diagram}
\usepackage{graphicx}

\usepackage{url}

\theoremstyle{plain}
\newtheorem{thm}{Theorem}[section]
\newtheorem{lem}[thm]{Lemma}

\newtheorem{prop}[thm]{Proposition}
\newtheorem{cor}[thm]{Corollary}

\newtheorem{dfn-lem}[thm]{Definition-Lemma}
\newtheorem{ex-thm}[thm]{Example}
\newtheorem*{thm*}{Theorem}

\theoremstyle{definition}
\newtheorem{dfn}[thm]{Definition}
\newtheorem{ex}[thm]{Example}

\newtheorem{rmk}[thm]{Remark}

\theoremstyle{remark}

\newcommand{\Z}{\mathbb{Z}}
\newcommand{\Q}{\mathbb{Q}}
\newcommand{\R}{\mathbb{R}}
\newcommand{\C}{\mathbb{C}}

\newcommand{\bmk}{\bm{k}}

\newcommand{\mfH}{\mathfrak{H}}

\newcommand{\mfa}{\mathfrak{a}}

\newcommand{\re}{\operatorname{Re}}

\newcommand{\isomto}{\stackrel{\sim}{\longrightarrow}}

\newcommand{\bs}{\backslash}

\newcommand{\ra}{\rightarrow}

\newcommand{\tp}[1]{{}^t\!#1} 
\newcommand{\setm}{\!-\!}

\numberwithin{equation}{section}

\newcommand{\PPC}{\mathbb{P}^{{n}-1}(\mathbb{C})}
\newcommand{\Rpos}{\mathbb{R}_{>0}}

\newcommand{\brk}[1]{\langle#1\rangle}

\begin{document}
\title{On the conical zeta values and the Dedekind zeta values for totally real fields}
\author{Hohto Bekki}
\address{Department of Mathematics, Faculty of Science and Technology, Keio University, 3-14-1 Hiyoshi, Kohoku-ku, Yokohama, Kanagawa, 223-8522, Japan}
\curraddr{}
\email{bekki@math.keio.ac.jp}
\thanks{}
\dedicatory{}

\begin{abstract}
The conical zeta values are a generalization of the multiple zeta values which are defined by certain multiple sums over convex cones. 
In this paper, we present a relation between the values of the Dedekind zeta functions for totally real fields and the conical zeta values for certain algebraic cones. 
More precisely, we show that the values of the partial zeta functions for totally real fields can be expressed as a rational linear combination of the conical zeta values associated with certain algebraic cones up to the square root of the discriminant. 
\end{abstract}

\maketitle


\section{Introduction}

\subsection{The conical zeta values}
Let ${n} \geq 1$ be an integer. 
For a subset $C \subset \R_{>0}^{n}$ and a multi-index $\bm k=(k_1, \dots, k_{n}) \in \Z_{\geq 1}^{n}$, set
\begin{align*}
\zeta_C(\bm k) := \sum_{x \in C\cap \Z^{n}} \frac{1}{x^{\bm k}}
\end{align*}
whenever the sum is convergent, where $x^{\bm k}:= x_1^{k_1}\cdots x_{n}^{k_{n}}$ for  $x=\tp(x_1, \dots, x_{n}) \in C \cap \Z^{n}$. 
In the case where $C \subset \R_{>0}^{n}$ is a cone (cf.~Definition \ref{dfn cone}), the value $\zeta_C(\bm k)$ is often called the \textit{conical zeta value} associated with the cone $C$, cf.~\cite{gpz}, \cite{terasoma}.

The following are the basic examples of the conical zeta values and are the motivation for the definition. 

\begin{ex}
Let $e_1, \dots, e_{n}$ denote the standard basis of $\R^{n}$, i.e., $e_i=\tp(0, \dots, 0, \overset{i}{1}, 0, \dots, 0)$. 
\begin{enumerate}
\item
If $C=\sum_{i=1}^{n} \Rpos e_i = \Rpos^{n}$, then 
\[
\zeta_C(\bm k) = \sum_{x \in \Z_{>0}^{n}}\frac{1}{x^{\bm k}}=\zeta(k_1)\cdots \zeta(k_{n})
\]
is the product of Riemann zeta functions. 
\item 
If $C =\sum_{i=1}^{n}\Rpos \sum_{j=i}^{n} e_j =\Rpos (e_1+\cdots +e_{n}) +\Rpos (e_2+\cdots +e_{n})  +\dots  +\Rpos e_{n}$, then 
\[
\zeta_C(\bmk) = \sum_{x_1<x_2< \cdots <x_{n}} \frac{1}{x^{\bm k}} =\zeta (k_1, \dots, k_{n})
\]
is the multiple zeta value. 
\end{enumerate}
\end{ex}

More generally, for any rational cone $C \subset \Rpos^{n}$, i.e., $C$ is of the form
\begin{align*}
C=\sum_{i=1}^r \Rpos \alpha_i  
\end{align*}
with some rational vectors $\alpha_1, \dots, \alpha_r \in \Q^{n}$, the following is known about the values $\zeta_C(\bm k)$.

\begin{thm}[Terasoma~\cite{terasoma}]
For a rational cone $C \subset \Rpos^{n}$, the conical zeta value $\zeta_C(\bm k)$ can be written as a $\Q^{ab}$-linear combination of the cyclotomic multiple zeta values. 
\end{thm}

On the other hand, to the best of the author's knowledge, it seems that little is known about the arithmetic properties of the conical zeta values associated with non-rational cones. 
In this paper, we consider the conical zeta values for certain algebraic cones which are not necessarily rational, and show that certain $\Q$-linear combinations of such conical zeta values describe the values of the partial zeta functions of totally real fields (up to the square root of the discriminant). 
Here, we say that a cone $C \subset \Rpos^{n}$ is algebraic if it is of the form
\begin{align*}
C=\sum_{i=1}^r \Rpos \alpha_i  
\end{align*}
with some algebraic vectors $\alpha_1, \dots, \alpha_r \in (\R \cap \overline{\Q})^{n}$.

\subsection{Main result}
More precisely, the following is the main theorem of this paper. 

\begin{thm}[cf.~Theorem \ref{thm main}]
Let $F$ be a totally real number field of degree ${n} \geq 1$, and let $\mfa \subset F$ be a fractional ideal of $F$. Then there exist a finite number of algebraic cones $C_1, \dots, C_m \subset \Rpos^{n}$ such that for any $k \in \Z_{\geq 2}$, we have
\begin{align*}
\zeta_{F,+}(\mfa^{-1}, k) \in \frac{1}{\sqrt{d_F}} \sum_{i=1}^m \sum_{\substack{\bm k \in \Z_{\geq 1}^{n}\\ |\bm k|={n}k}} \Q \zeta_{C_i}(\bm k), 
\end{align*}
where $\zeta_{F,+}(\mfa^{-1},s)$ is the (narrow) partial zeta function associated with $\mfa^{-1}$, $d_F$ is the discriminant of $F$, and $|\bm k|:=k_1+\cdots +k_{n}$ for $\bm k =(k_1, \dots, k_{n})$. 
\end{thm}

Actually, the algebraic cones $C_1, \dots, C_m \subset \Rpos^{n}$ and the coefficients of each $\zeta_{C_i}(\bm k)$ can be computed by using the so-called Shintani's cone decomposition. 

\begin{ex}[cf.~Section \ref{sec ex sqrt5}]\label{ex sqrt5}
Let $F=\Q(\sqrt{5})$, and let $\mfa =\mathcal O_F=\Z[\frac{1+\sqrt{5}}{2}]$. 
Moreover, let 
\begin{align*}
C&= 
\Rpos 
\begin{pmatrix}
\frac{3+\sqrt{5}}{2} \\
1 \\
\end{pmatrix}
+\Rpos
\begin{pmatrix}
\frac{3-\sqrt{5}}{2} \\
1 \\
\end{pmatrix} \\
&=
\{
\tp(x_1, x_2) \in \Rpos^2 \mid
-x_1^2+3x_1x_2-x_2^2>0
\}
\end{align*}
be an algebraic cone generated by $\tp \left(\frac{3+\sqrt{5}}{2}, 1 \right)$ and $\tp \left(\frac{3-\sqrt{5}}{2}, 1 \right)$. 
Then we can prove that 
\begin{align}\label{eqn zeta}
\zeta_{\Q(\sqrt{5})}(k)=\zeta_{F,+}(\mfa^{-1}, k) \in \frac{1}{\sqrt{5}} \sum_{\substack{\bm (k_1, k_2) \in \Z_{\geq 1}^2\\ k_1+k_2=2k}} \Q \zeta_{C}(k_1, k_2)
\end{align}
for all $k \in \Z_{\geq 2}$, where $\zeta_{\Q(\sqrt{5})}(s)$ is the Dedekind zeta function for $\Q(\sqrt{5})$.  
For example, we find
\begin{align}
\zeta_{\Q(\sqrt{5})}(2) &=
 \frac{1}{5\sqrt{5}}(4\zeta_C(3,1)+3\zeta_C(2,2)), \label{eqn zeta2}\\
\zeta_{\Q(\sqrt{5})}(3) &=
 \frac{1}{25\sqrt{5}}(12\zeta_C(5,1)+18\zeta_C(4,2)+11\zeta_C(3,3)). \label{eqn zeta3}
\end{align}
\end{ex}

\begin{rmk}
Note that the theorem of Siegel-Klingen describes all the critical values of the partial zeta functions in a more beautiful way, e.g., $\zeta_{\Q(\sqrt{5})}(2) = \frac{2 \pi^4}{75\sqrt{5}}$. On the other hand, one feature of the above theorem is that it describes both critical and non-critical values in a uniform way using the conical zeta values associated with algebraic cones. 
\end{rmk}

In the next section (Section \ref{sec main}), after fixing some notation and convention, we state our main theorem (in a slightly more precise way), and then prove the main theorem. 
The proof is actually very simple and elementary. We start from an integral representation of the values of the partial zeta functions which can be seen as a variant of the classical Hecke integral formula, that is, a formula which expresses the values of the partial zeta functions of real quadratic fields as an integral of the Eisenstein series along a closed geodesics on the modular curve. 
One key point of the integral representation in this paper is that we consider a ``partial'' Eisenstein series which can be seen as a decomposed piece of the Eisenstein series along cones (cf.~Remark \ref{rmk psi} (1)). 
Then by using some other also classical techniques such as the unfolding of the integrals and the Shintani cone decomposition, we prove our theorem. 
In Section \ref{sec example}, we present some examples to illustrate our main theorem.

\subsection*{Acknowledgments}
I would like to express my gratitude to Kenichi Bannai for the constant encouragement and valuable comments during the study. 
This work was supported by JSPS KAKENHI Grant Number JP20J01008.

\section{Main Theorem}\label{sec main}

\paragraph{Convention}
\begin{itemize}
\item We fix an integer ${n} \geq 1$ throughout the paper.
\item For a matrix $A$, its transpose is denote by $\tp A$.
\item For a ring $R$, elements in $R^r$ are basically regarded as column vectors, and the matrix algebra $M_r(R)$ acts on $R^r$ by the matrix multiplication form the left.
\item For vectors $v_1, \dots, v_{r'} \in R^r$, we often regard the $r'$-tuple $I=(v_1, \dots, v_{r'})\in (R^r)^{r'}$ as an $r\times r'$-matrix whose columns are $v_1, \dots, v_{r'}$. 
\item For vectors $x=\tp(x_1, \dots, x_{n}), y=\tp(y_1, \dots, y_{n})$, the bracket $\brk{x,y}:=x_1y_1+\cdots +x_{n}y_{n}$ denotes the dot product of $x$ and $y$. 
\end{itemize}

\subsection{Cones and conical zeta values}

First we fix some notation and terminologies concerning cones that will be used in this paper.

For $r \geq 0$, $I=(\alpha_1, \dots, \alpha_r) \in (\R^{n} \setm \{0\})^r$, we set
\begin{align}\label{eqn cone}
C_I:= \sum_{i=1}^r \Rpos \alpha_i \subset \R^{n}. 
\end{align}
In the case where $r=0, I=\emptyset$, we set $C_{\emptyset}:=\{0\}$.

\begin{dfn}\label{dfn cone}
\begin{enumerate}
\item 
A subset $C \subset \R^{n}$ is called an \textit{open convex polyhedral cone} if $C$ is of the form $C=C_I$ for some $r \geq 0$, $I=(\alpha_1, \dots, \alpha_r) \in (\R^{n} \setm \{0\})^r$. In this case, we say that $C$ is generated by $\alpha_1, \dots, \alpha_r$. Since in this paper we basically deal only with  {open convex polyhedral cones}, by abuse of notation, we will refer to an {open convex polyhedral cone} simply as a \textit{cone}. 
\item 
A cone $C \subset \R^{n}$ is said to be \textit{simplicial} if we can take linearly independent generators of $C$, i.e., there exists $I=(\alpha_1, \dots, \alpha_r) \in (\R^{n} \setm \{0\})^r$ such that $C=C_I$ and $\alpha_1, \dots, \alpha_r$ are linearly independent over $\R$. 
\item 
A cone $C \subset \R^{n}$ is said to be \textit{rational} if we can take rational generators of $C$, i.e., there exists $I=(\alpha_1, \dots, \alpha_r) \in (\Q^{n} \setm \{0\})^r$ such that $C=C_I$. 
\item 
More generally, let $K \subset \R$ be a subfield. Then a cone $C \subset \R^{n}$ is said to be \textit{$K$-rational} if there exists $I=(\alpha_1, \dots, \alpha_r) \in (K^{n} \setm \{0\})^r$ such that $C=C_I$. 
\item 
A cone $C \subset \R^{n}$ is said to be \textit{algebraic} if there exists an algebraic subfield $K \subset \R \cap \overline{\Q}$ such that $C$ is $K$-rational. 
\item 
A cone $C \subset \R^{n}$ is said to be \textit{smooth} if there exists $I=(\alpha_1, \dots, \alpha_r) \in (\Z^{n} \setm \{0\})^r$ which can be extended to a basis of $\Z^{n}$ such that $C=C_I$. In particular, a smooth cone is a rational simplicial cone. 
\item 
A cone $C \subset \R^{n}$ is said to be \textit{totally positive} if $C \subset \Rpos^{n}$. 
\end{enumerate}
\end{dfn}

\begin{dfn}
Let $C \subset \Rpos^{n}$ be a totally positive (open convex polyhedral) cone, and let $\bm k=(k_1, \dots, k_{n}) \in \Z_{\geq 1}^{n}$ be a multi-index. Then we define the conical zeta value associated with the cone $C$ with index $\bm k$ to be
\begin{align*}
\zeta_C(\bm k) := \sum_{x \in C\cap \Z^{n}} \frac{1}{x^{\bm k}}
\end{align*}
whenever the sum is convergent, where $x^{\bm k}:= x_1^{k_1}\cdots x_{n}^{k_{n}}$ for  $x=\tp(x_1, \dots, x_{n}) \in C \cap \Z^{n}$. 
\end{dfn}

\subsection{Statement of the main theorem}

Let $F$ be a totally real field of degree ${n}$, and let
\[
\tau_1, \dots, \tau_{n} \colon F \hookrightarrow \R
\]
be the field embeddings of $F$ into $\R$. 
Moreover, we put $F_{\R}:= F \otimes_{\Q}\R$. Then $\tau_1, \dots, \tau_{n}$ induces an isomorphism
\[
\tau=(\tau_1, \dots, \tau_{n}) \colon F_{\R} \isomto \R^{n}. 
\]
For a subset $A \subset F_{\R}$, we denote by $A_+$ its totally positive part, i.e., 
\[
A_+:=\{ a \in A \mid \forall i \in \{1, \dots, n\}, \tau_i(a)>0\}. 
\]

Let $\mathcal O_F$ denotes the ring of integers of $F$, and let $\mfa \subset F$ be a fractional ideal of $F$.  
Then we define the (narrow) partial zeta function associated with $\mfa^{-1}$ to be 
\begin{align*}
\zeta_{F,+}(\mfa^{-1}, s) := \sum_{x \in \mfa_{+}/\mathcal O_{F,+}^{\times}} \frac{1}{N_{F/\Q}(x)^s}
\end{align*}
for $\re(s)>1$, where $N_{F/\Q}$ is the norm of the extension $F/\Q$.

Finally, we define
\[
F' := \tau_1(F)\cdots \tau_{n}(F) \subset \R
\]
to be the subfield of $\R$ generated by $\tau_1(F), \dots, \tau_{n}(F)$.

The following is the main theorem of this paper. 
\begin{thm}\label{thm main}
Let $F$ and $\mfa$ be as above. Then there exist a finite number of totally positive $F'$-rational cones $C_1, \dots, C_m \subset \Rpos^{n}$ such that for any $k \in \Z_{\geq 2}$, we have
\begin{align}\label{eqn main thm}
\zeta_{F,+}(\mfa^{-1}, k) \in \frac{1}{\sqrt{d_F}} \sum_{i=1}^m \sum_{\substack{\bm k \in \Z_{\geq 1}^{n}\\ |\bm k|={n}k}} \Q \zeta_{C_i}(\bm k), 
\end{align}
where $d_F$ is the discriminant of $F$ and $|\bm k|:=k_1+\cdots +k_{n}$ for $\bm k =(k_1, \dots, k_{n})$ is the weight of the multi-index. In other words, the value $\zeta_{F,+}(\mfa^{-1}, k)$ of the partial zeta function at $k$ can be written as a rational linear combination of the conical zeta values associated with $F'$-rational cones with indices of weight $nk$ divided by the square root of the discriminant. 
\end{thm}

\begin{rmk}
The $F'$-rational cones $C_1, \dots, C_m \subset \Rpos^{n}$ and the coefficients of $\zeta_{C_i}(\bm k)$ in \eqref{eqn main thm} can be computed using the Shintani cone decomposition (cf.~\eqref{eqn main formula} and Section \ref{sec example}), but they are not necessarily unique.
\end{rmk}

\begin{rmk}
This theorem shows the connection between the conical zeta values associated with algebraic cones and the values of the partial zeta functions of totally real fields. On the other hand, it does not tell us about the properties of individual conical zeta values. It might be an interesting problem to consider the arithmetic properties of the individual conical zeta values associated with algebraic cones. 
\end{rmk}

The proof of this theorem will be given in Section \ref{sec proof}. 
To this end, we first prepare an integral representation of the values of the partial zeta functions (Proposition \ref{prop int rep}). 
We will use the similar arguments and notations as in \cite[Section 2, Section 7]{bekki}.

\subsection{The Hecke type integral representation}

Let us take a basis $w_1, \dots, w_{n} \in \mfa$ of $\mfa$ over $\Z$, and set
\begin{align*}
w &:=\tp(w_1, \dots, w_{n}) \in \mfa^n \subset F^n, \\
w^{(i)} &:= \tau_i(w)=\tp(\tau_i(w_1), \dots, \tau_i(w_{n})) \in \R^n 
\end{align*}
for $i=1, \dots, {n}$, and 
\begin{align*}
N_w(x):= \prod_{i=1}^{n}\brk{x, w^{(i)}} \in \Q[x_1, \dots, x_n]
\end{align*}
for $x=\tp(x_1, \dots, x_{n})$. 
By replacing $w_1$ with $-w_1$ if necessary, we assume $\det (w^{(1)}, \dots, w^{(n)})>0$, where $(w^{(1)}, \dots, w^{(n)})$ is regarded as an $n\times n$-matrix. 
Then we have 
\begin{align*}
\det (w^{(1)}, \dots, w^{(n)})=\sqrt{d_{F}} N\mfa, 
\end{align*}
where $d_F$ is the discriminant of $F$ and $N\mfa$ is the norm of the fractional ideal $\mfa$. 
Note also that $w$ defines an isomorphisms 
\begin{align}\label{eqn basis}
\vcenter{
\xymatrix@R=0pt{
{\brk{-,w}\colon} \Q^n \ar[r]^-{\sim}&F {;\  x \mapsto \brk{x,w}} \\
\phantom{\brk{-,w}\colon} \cup&\cup \phantom{;\  x \mapsto \brk{x,w}} \\
\phantom{\brk{-,w}\colon} \Z^n \ar[r]^-{\sim}&\mfa \phantom{;\  x \mapsto \brk{x,w}} \\
}
}
\end{align}
which extends also to ${\brk{-,w}\colon} \R^{n}  \isomto F_{\R}$, and $N_w$ is the norm map with respect to this isomorphism, i.e., we have $N_w(x)=N_{F/\Q}(\brk{x,w})$ for $x \in \R^n$.  
We define
\begin{align*}
  T_{w,+}:=&\{x \in \R^{n} \mid \brk{x,w} \in F_{\R, +}\}\\
  =& \{x \in \R^n \mid \forall i \in \{1, \dots, n\}, \brk{x, w^{(i)}}>0 \}
\end{align*}
to be the subset of $\R^{n}$ corresponding to the set $F_{\R,+}$ of totally positive elements in $F_{\R}$ under this isomorphism. 

Moreover, let
\begin{align*}
\rho_w \colon F^{\times} \ra GL_{n}(\Q)
\end{align*}
be the regular representation of $F^{\times}$ on $F$ with respect to the isomorphism \eqref{eqn basis}, i.e., 
\begin{align*}
\brk{\rho_w(\alpha)x, w}= \alpha \brk{x,w} \in F
\end{align*}
for $\alpha \in F^{\times}$ and $x \in \Q^{n}$, and set 
\begin{align*}
\Gamma_{w,+}:= \rho_w(\mathcal O_{F,+}^{\times}) \subset SL_{n}(\Z)
\end{align*}
to be the subgroup of $SL_n(\Z)$ corresponding to the totally positive units of $\mathcal O_F$ under the map $\rho_w$. 
Note that $\Gamma_{w,+}$ acts on $T_{w,+}\cap \Z^{n}$. 

Now, let $w^*_1, \dots, w^*_{n} \in F$ be the dual basis of $w_1, \dots, w_{n}$ with respect to the trace $Tr_{F/\Q}$ of the field extension $F/\Q$, i.e., 
\begin{align*}
Tr_{F/\Q}(w_i w^*_j) = \delta_{ij} 
\end{align*}
for $i,j=1, \dots, {n}$, where $\delta_{ij}$ is the Kronecker delta. 
Then it is known that $w^*_1, \dots, w^*_{n}$ form a basis of the fractional ideal 
\[
\mfa^* = \{ \alpha \in F \mid Tr_{F/\Q}(\alpha \mfa) \subset \Z\} \subset F.   
\]
Repeating the above construction, we define
\begin{align*}
w^* &:=\tp(w^*_1, \dots, w^*_{n}) \in F^n, \\
w^{*(i)} &:= \tau_i(w^*)=\tp(\tau_i(w^*_1), \dots, \tau_i(w^*_{n})) \in \R^n, \\
N_{w^*}(x) &:= \prod_{i=1}^n \brk{x, w^{*(i)}} \in \Q[x_1, \dots, x_{n}], \\
\rho_{w^*} &\colon F^{\times} \ra GL_{n}(\Q) 
\end{align*}
starting from the dual basis $w^*_1, \dots, w^*_{n}$.

\begin{lem}\label{lem basic}
  \begin{enumerate}
  \item
    $w^{*(1)}, \dots, w^{*(n)}$ are the dual basis of $w^{(1)}, \dots, w^{(n)}$ with respect to $\brk{-,-}$, i.e., we have
    \[
    \brk{w^{(i)}, w^{*(j)}}=\delta_{ij}.
    \]\label{item dual basis}
  \item
    We have
    \begin{align*}
      C_{(w^{(1)}, \dots, w^{({n})})}
      &=\{x \in \R^n \mid \forall i \in \{1, \dots, n\}, \brk{x,w^{*(i)}}>0\}, \\
      C_{(w^{*(1)}, \dots, w^{*({n})})}
    &=\{x \in \R^n \mid \forall i \in \{1, \dots, n\}, \brk{x,w^{(i)}}>0\}=T_{w,+}.
    \end{align*}
    Here recall that $C_{(w^{(1)}, \dots, w^{({n})})}$ (resp.~$C_{(w^{*(1)}, \dots, w^{*({n})})}$) denotes the cone generated by the vectors $w^{(1)}, \dots, w^{({n})}$ (resp.~$w^{*(1)}, \dots, w^{*({n})}$), cf.~\eqref{eqn cone}.
    In particular, we see that $T_{w,+}$ is an $F'$-rational cone. \label{item dual cone}
  \item
    For $\alpha \in F^{\times}$ and $i \in \{1, \dots, n\}$, we have
    \[
    \rho_w(\alpha)w^{*(i)}=\tau_i(\alpha)w^{*(i)}, 
    \]
    i.e., $w^{*(i)}$ is an eigenvector of $\rho_w(\alpha)$ with eigenvalue $\tau_i(\alpha)$.\label{item eigenvector}
  \item
    For $\gamma \in \Gamma_{w,+}$, we have
    \[
N_{w^*}(\tp\gamma x)=N_{w^*}(x). 
    \]\label{item norm}
\end{enumerate}
\end{lem}
\begin{proof}
  (1) Set
  \begin{align*}
    W&:=(w^{(1)}, \dots, w^{(n)}) \in GL_n(\R), \\
    W^*&:=(w^{*(1)}, \dots, w^{*(n)}) \in GL_n(\R).
  \end{align*}
  Then since $w_1^*, \dots, w_n^*$ are the dual basis of $w_1, \dots, w_n$ with respect to $Tr_{F/\Q}$, we find
  \[
  W^* \tp W =(Tr_{F/\Q}(w_i^*w_j))_{ij}=(\delta_{ij})_{ij},
  \]
  and hence
  \[
  (\brk{w^{(i)}, w^{*(j)}})_{ij}=\tp W W^*=(\delta_{ij})_{ij}. 
  \]
  
%
  The assertions (2) and (3) follow from (1), and (4) follows from (3). 
\end{proof}

Next, we consider the following infinite series. 
\begin{dfn}
For $k \in \Z_{\geq 1}$, a subset $C \subset \R^{n}\setm \{0\}$, and $y \in \C^{n} \setm \{0\}$, set
\begin{align*}
\psi_{nk, C}(y) := \sum_{x \in C \cap \Z^{n}} \frac{1}{\brk{x,y}^{{n}+{n}k}} 
\end{align*}
whenever the sum is convergent. 
\end{dfn}

\begin{rmk}\label{rmk psi}
\begin{enumerate}
\item 
In the case where $n=2$, $y=(z,1)$ with $z \in \mfH$ (the upper half plane), the series 
\[
\psi_{2k, C}(y) = \sum_{\tp(x_1, x_2) \in C \cap \Z^{2}} \frac{1}{(x_1z+x_2)^{{2}+{2}k}} 
\]
can be seen as a ``fragment'' of the holomorphic Eisenstein series of weight $2+2k$, in which the sum is restricted to the subset $C$. 
\item 
In \cite{bekki}, we have considered the case where $C$ is (the $Q$-perturbation of) a rational cone, in which case $\psi_{nk, C}$ is essentially a finite sum of the Barnes zeta functions. In this paper, we will consider the case where $C$ is an algebraic cone $T_{w,+}$. 
\end{enumerate}
\end{rmk}

\begin{prop}\label{prop psi}
For $k \in \Z_{\geq 1}$, the infinite series
\begin{align*}
\psi_{nk, T_{w,+}}(y) = \sum_{x \in T_{w,+}\cap \Z^{n}} \frac{1}{\brk{x,y}^{{n}+{n}k}}. 
\end{align*}
converges absolutely and locally uniformly on
\begin{align*}
\widetilde{Y}_{w,+}:=
\{y \in \C^{n} \setm \{0\} \mid \exists \lambda \in \C^{\times}, \forall i \in \{1, \dots, n\}, \re (\brk{w^{*(i)}, \lambda y}) >0\}. 
\end{align*}
\end{prop}

\begin{proof}
This can be proved in a straightforward way. Here we give a proof which uses \cite[Proposition 6.1.2]{bekki}. 
First, note that it suffices to prove that $\psi_{nk, T_{w,+}}(y)$ converges absolutely and locally uniformly on
\begin{align*}
V:=
\{y \in \C^{n} \setm \{0\} \mid \forall i \in \{1, \dots, n\}, \re (\brk{w^{*(i)}, y}) >0\}. 
\end{align*}
Let $y'=\tp(y'_1, \dots, y'_{n}) \in V$. By Lemma \ref{lem basic} (\ref{item dual cone}), this is equivalent to 
\[
\re(y')=\tp(\re (y'_1), \dots, \re(y'_{n})) \in C_{(w^{(1)}, \dots, w^{({n})})}=\sum_{i=1}^{n}\Rpos w^{(i)}. 
\]
Then we can take linearly independent $\alpha_1, \dots, \alpha_{n} \in \Q^n \cap C_{(w^{(1)}, \dots, w^{({n})})}$ such that $\re(y') \in C_{(\alpha, \dots, \alpha_{n})}$. 
Let $\alpha_1', \dots, \alpha_{n}' \in \Q^n$ be the dual basis of $\alpha_1, \dots, \alpha_{n}$ with respect to the dot product $\brk{-,-}$. 
Set $I:=(\alpha_1', \dots, \alpha_{n}')$, and 
\begin{align*}
  V_I
  :=&
  \{y \in \C^{n} \setm \{0\} \mid \forall i \in \{1, \dots, n\}, \re (\brk{\alpha_{i}', y}) >0\} \\
  =&
  \{y \in \C^n \setm \{0\} \mid \re(y) \in C_{(\alpha, \dots, \alpha_{n})}\}. 
\end{align*}
Then again by using Lemma \ref{lem basic} (2), we see that $y' \in V_I \subset V$ and that $T_{w,+} \subset C_I=\sum_{i=1}^{n}\Rpos \alpha_i'$. 
Therefore, now it suffices to prove that
\begin{align*}
\psi_{nk, C_I}(y) = \sum_{x \in C_I \cap \Z^{n}} \frac{1}{\brk{x,y}^{{n}+{n}k}} 
\end{align*}
converges absolutely and locally uniformly on $V_I$, and this follows from \cite[Proposition 6.1.2]{bekki}. 
\end{proof}

Let $\omega$ denotes the $(n-1)$-form on $\C^{n}\setm \{0\}$ defined by
\begin{align*}
\omega(y):= \sum_{i=1}^{n} (-1)^{i-1} y_i dy_1 \wedge \cdots \wedge \check{dy_i} \wedge \cdots \wedge dy_n, 
\end{align*}
where $\check{dy_i}$ means that $dy_i$ is omitted. 
Moreover, let 
\begin{align*}
\pi_{\C}: \C^{n} \setm \{0\} \ra \PPC:= (\C^{n} \setm \{0\})/\C^{\times}
\end{align*}
denotes the natural projection, and set
\begin{align*}
Y_{w,+}:=\pi_{\C}(\widetilde{Y}_{w,+}) \subset \PPC. 
\end{align*}
Note that by Lemma \ref{lem basic} (\ref{item dual cone}), (\ref{item eigenvector}), we see that $\Gamma_{w, +} \subset SL_n(\Z)$ acts on $\widetilde{Y}_{w,+}$ and $Y_{w,+}$ by
\begin{align}\label{eqn action y}
\tp\gamma^{-1}: \widetilde{Y}_{w,+}\isomto \widetilde{Y}_{w,+}; \ y \mapsto \tp\gamma^{-1}y
\end{align}
where $\tp\gamma^{-1}y$ on the right hand side is the usual matrix action of the transposed inverse of $\gamma$ from the left.

\begin{cor} 
For $k \in \Z_{\geq 1}$, the $(n-1)$-form
\begin{align*}
N_{w^*}(y)^{k}\psi_{nk, T_{w,+}}(y) \omega(y)
\end{align*}
on $\widetilde{Y}_{w,+}$ defines a $\Gamma_{w,+}$-invariant closed $(n-1)$-form on $Y_{w,+} \subset \PPC$. 
Here the $\Gamma_{w,+}$-invariance means that we have
\[
N_{w^*}(\tp\gamma^{-1}y)^{k}\psi_{nk, T_{w,+}}(\tp\gamma^{-1}y) \omega(\tp\gamma^{-1}y)
=
N_{w^*}(y)^{k}\psi_{nk, T_{w,+}}(y) \omega(y)
\]
for all $\gamma \in \Gamma_{w,+}$. 
\end{cor}

\begin{proof}
  By Proposition \ref{prop psi}, $N_{w^*}(y)^{k}\psi_{nk, T_{w,+}}(y)$ defines a holomorphic function on $\widetilde{Y}_{w,+}$ which is homogeneous of degree $-n$, i.e.,
  \[
  N_{w^*}(\lambda y)^{k}\psi_{nk, T_{w,+}}(\lambda y)=\lambda^{-n}N_{w^*}(y)^{k}\psi_{nk, T_{w,+}}(y)
  \]
  for $\lambda \in \C^{\times}$. Therefore, we see that $N_{w^*}(y)^{k}\psi_{nk, T_{w,+}}(y)\omega(y)$ defines a holomorphic closed $(n-1)$-form on $Y_{w,+} \subset \PPC$. 
  
  To see the $\Gamma_{w,+}$-invariance, let $\gamma \in \Gamma_{w,+}$. First by Lemma \ref{lem basic} (\ref{item norm}), we have $N_{w^*}(\tp\gamma^{-1}y)^{k}=N_{w^*}(y)^{k}$. Moreover, since we see that $\omega(gy)=\det(g)\omega(y)$ for $g \in GL_n(\C)$, we have $\omega(\tp \gamma^{-1}y)=\omega(y)$.
  Finally, since $\Gamma_{w,+}$ is acting on $T_{w,+}\cap \Z^n$, we find
  \begin{align*}
    \psi_{nk, T_{w,+}}(\tp\gamma^{-1}y)
    &=
    \sum_{x \in T_{w,+}\cap \Z^n} \frac{1}{\brk{x, \tp \gamma^{-1}y}^{n+nk}}\\
    &=
    \sum_{x \in T_{w,+}\cap \Z^n} \frac{1}{\brk{\gamma^{-1}x, y}^{n+nk}}\\
    &=
    \sum_{x \in T_{w,+}\cap \Z^n} \frac{1}{\brk{x, y}^{n+nk}}=\psi_{nk,T_{w,+}}(y).\\
  \end{align*}
  This shows the corollary. 
\end{proof}

Now, set
\begin{align*}
\Delta^{\circ}_{n-1}:=\Bigg\{t=\tp(t_1, \dots, t_{n}) \in \Rpos^n \ \Bigg| \  \sum_{i=1}^{n}t_i=1\Bigg\}. 
\end{align*}
We embed $\Delta^{\circ}_{n-1}$ into $\R^{n-1}$ by
\[
\Delta^{\circ}_{n-1} \hookrightarrow \R^{n-1}; \tp(t_1, \dots, t_n) \mapsto \tp(t_2, \dots, t_n)
\]
and equip $\Delta^{\circ}_{n-1}$ with an orientation induced from the standard orientation on $\R^{n-1}$. 

For $I=(\alpha_1, \dots, \alpha_n) \in (\C^n\setm \{0\})^n$ such that $\alpha_1, \dots, \alpha_n$ are a basis of $\C^n$ over $\C$, we define
\begin{align*}
\Delta^{\circ}_{I}:=\pi_{\C} \left(C_{(\alpha_1, \dots, \alpha_n)}\right) \subset \PPC 
\end{align*}
to be the image of the cone $C_{(\alpha_1, \dots, \alpha_n)}=\sum_{i=1}^{n}\Rpos \alpha_i \subset \C^n \setm \{0\}$ in $\PPC$. 
We have an isomorphism
\begin{align*}
\sigma_{I}: \Delta^{\circ}_{n-1} \isomto \Delta^{\circ}_{I} \subset \PPC;\ t=\tp(t_1, \dots, t_{n}) \mapsto \pi_{\C}\Bigg(\sum_{i=1}^{n}t_i \alpha_i\Bigg), 
\end{align*}
and we equip $\Delta^{\circ}_{I}$ with an orientation induced from $\Delta^{\circ}_{n-1}$ via this isomorphism.

Let us consider the case where $\alpha_i=w^{(i)}$ and set
\[
\Delta^{\circ}_{w,+}:=\Delta^{\circ}_{(w^{(1)}, \dots, w^{(n)})}=\pi_{\C}(C_{(w^{(1)}, \dots, w^{(n)})}) \subset \PPC. 
\]
By Lemma \ref{lem basic} (\ref{item dual cone}), we have $C_{(w^{(1)}, \dots, w^{(n)})} \subset \widetilde{Y}_{w,+}$, and hence $\Delta^{\circ}_{w,+} \subset Y_{w,+}$. 
Moreover, by Lemma \ref{lem basic} (\ref{item eigenvector}), we see that the action \eqref{eqn action y} of $\Gamma_{w,+}$ on $\widetilde{Y}_{w,+}$ (resp.~$Y_{w,+}$) preserves $C_{(w^{(1)}, \dots, w^{({n})})}$ (resp.~$\Delta^{\circ}_{w,+}$ and its orientation), and hence by Dirichlet's unit theorem we see that $\Gamma_{w,+}\bs \Delta^{\circ}_{w,+}$ is a compact oriented manifold of dimension $n-1$.

Then we have the following integral representation of $\zeta_{F,+}(\mfa^{-1}, k)$.

\begin{prop}\label{prop int rep}
For $k \in \Z_{\geq 2}$, we have
\begin{align}\label{eqn int formula}
\int_{\Gamma_{w,+}\bs \Delta^{\circ}_{w,+}} 
N_{w^*}(y)^{k-1}\psi_{n(k-1), T_{w,+}}(y) \omega(y)
=
\frac{((k-1)!)^n \sqrt{d_{F}} N\mfa}{(nk-1)!}
\zeta_{F,+}(\mfa^{-1}, k).  
\end{align}
\end{prop}

In order to prove this proposition, we recall a classical formula known as the Feynman parametrization. 

\begin{prop}[Feynman parametrization]\label{prop feynman}
Let $x \in \C^n \setm \{0\}$, and let $\alpha_1, \dots, \alpha_n \in \C^{n}$ be a basis over $\C$ such that $\re(\brk{x, \alpha_i})>0$ for all $i=1, \dots, n$. Moreover, let $\alpha'_1, \dots, \alpha'_{n} \in \C^n$ be the dual basis of $\alpha_1, \dots, \alpha_n$ with respect to $\brk{-,-}$. 
Set 
\[
\Delta^{\circ}_{(\alpha_1, \dots, \alpha_n)}:=\pi_{\C}(C_{(\alpha_1, \dots, \alpha_n)}) \subset \PPC, 
\]
and equip $\Delta^{\circ}_{(\alpha_1, \dots, \alpha_n)}$ with an orientation as above. Then  for $\bm k=(k_1, \dots, k_n) \in \Z_{\geq 0}^n$, we have
\begin{align*}
\int_{\Delta^{\circ}_{(\alpha_1, \dots, \alpha_n)}} 
\brk{\alpha'_1,y}^{k_1}\cdots \brk{\alpha'_n,y}^{k_n} \frac{\omega(y)}{\brk{x,y}^{n+|\bm k|}}
=
\frac{\bm k!}{(n+|\bm k|-1)!} \frac{\det (\alpha_1, \dots, \alpha_n)}{\brk{x, \alpha_1}^{k_1+1} \cdots \brk{x, \alpha_n}^{k_n+1}},
\end{align*}
where $|\bm k|=k_1+\cdots +k_n$ and $\bm k!=k_1!\cdots k_n!$. 
\end{prop}

\begin{proof}
See, for example, \cite{hurwitz} or \cite[Proposition 7.1.3]{bekki}. 
\end{proof}

\begin{proof}[Proof of Proposition \ref{prop int rep}]
Let $A \subset T_{w,+}\cap \Z^n$ be a system of representatives of $\Gamma_{w,+}\bs (T_{w,+}\cap \Z^n)$. Then by using Proposition \ref{prop feynman} as well as the $\Gamma_{w,+}$-invariance of $N_{w*}$ (Lemma \ref{lem basic} (\ref{item norm})) and $\omega$, we find
\begin{align*}
&\int_{\Gamma_{w,+}\bs \Delta^{\circ}_{w,+}} 
N_{w^*}(y)^{k-1}\psi_{n(k-1), T_{w,+}}(y) \omega(y) \\
=&
\int_{\Gamma_{w,+}\bs \Delta^{\circ}_{w,+}} 
N_{w^*}(y)^{k-1}
\sum_{\gamma \in \Gamma_{w,+}} \sum_{x \in A} \frac{1}{\brk{\gamma x, y}^{nk}}
\omega(y) \\
=&
\sum_{x \in A}
\int_{\Gamma_{w,+}\bs \Delta^{\circ}_{w,+}} 
\sum_{\gamma \in \Gamma_{w,+}} 
N_{w^*}(\tp\gamma y)^{k-1}
\frac{1}{\brk{x, \tp \gamma y}^{nk}}
\omega(\tp \gamma y) \\
=&
\sum_{x \in A}
\int_{\Delta^{\circ}_{w,+}} 
N_{w^*}(y)^{k-1}
\frac{1}{\brk{x, y}^{nk}}
\omega(y)\\
=&
\frac{((k-1)!)^n \det(w^{(1)}, \dots, w^{({n})})}{(nk-1)!}
\sum_{x \in A}
\frac{1}{N_w(x)^k}\\
=&
\frac{((k-1)!)^n \sqrt{d_{F}} N\mfa}{(nk-1)!}
\zeta_{F,+}(\mfa^{-1}, k),  
\end{align*}
where Proposition \ref{prop feynman} is used in the fourth equality with $\alpha_i=w^{(i)}$. 
\end{proof}

\subsection{Proof of the main theorem}\label{sec proof}

We also need the following classical fact. 

\begin{prop}[Shintani cone decomposition]\label{prop shintani}
There exists a finite set
\[
\Phi \subset \coprod_{r=1}^{n} (C_{(w^{(1)}, \dots, w^{(n)})} \cap \Z^n)^r
\]
satisfying the following conditions: 
\begin{enumerate}
\item[(i)] 
For all $I=(\alpha_1, \dots, \alpha_r) \in \Phi$, the vectors $\alpha_1, \dots, \alpha_r \in C_{(w^{(1)}, \dots, w^{(n)})} \cap \Z^n$ can be extended to a basis of $\Z^n$ over $\Z$. In particular, the cone $C_I$ is a smooth cone for all $I \in \Phi$. 
\item[(ii)] 
We have
\[
C_{(w^{(1)}, \dots, w^{(n)})} = \coprod_{\gamma \in \Gamma_{w,+}} \coprod_{I \in \Phi} \tp \gamma C_I. 
\]
\end{enumerate}
Here $\coprod$ denotes the disjoint union. 
\end{prop}

\begin{proof}
It is well known that there exists $\Phi$ with condition (ii), cf. \cite{shintani}. 
Then by subdividing each cone if necessary, we achieve the condition (i), cf. \cite[Section 11.1]{cls}.  
\end{proof}

\begin{rmk}
Note that actually, there is also a stronger version of this proposition which require $\Phi$ to be a fan, cf. \cite[Chapter III, Corollary 7.6]{amrt}, \cite{ishida}. 
However, for our purpose, Proposition~\ref{prop shintani} is sufficient. 
\end{rmk}

\begin{proof}[Proof of Theorem \ref{thm main}]
We will prove the theorem by computing the left hand side of \eqref{eqn int formula} in a different way from Proposition \ref{prop int rep} using the Shintani cone decomposition.

Take $\Phi$ as in Proposition \ref{prop shintani}, and let 
\begin{align}\label{eqn n-dim}
\Phi^{(n)}:= \Phi \cap (C_{(w^{(1)}, \dots, w^{(n)})} \cap \Z^n)^n, 
\end{align}
i.e., $\Phi^{(n)}$ is exactly the subset of $\Phi$ such that $C_I$ is an $n$-dimensional cone for $I \in \Phi^{(n)}$.
By permuting the order of the vectors if necessary, we may assume 
\[
\det(I)=\det (\alpha_1, \dots, \alpha_n)=1
\]
for all $I=(\alpha_1, \dots, \alpha_n) \in \Phi^{(n)}$. 
Note that in the case $n=1$, we automatically have $\alpha_1=1$, because we have assumed that $\det(w^{(1)}, \dots, w^{(n)})=w^{(1)} >0$ and $\alpha_1 \in \Rpos w^{(1)}$.

For $I \in \Phi^{(n)}$, note that we have
\[
\Delta^{\circ}_I=\pi_{\C}\left( C_I \right) \subset \Delta_{w,+} \subset Y_{w,+},
\]
and that the orientation of $\Delta^{\circ}_I$ coincides with the orientation restricted from $\Delta^{\circ}_{w,+}$ because $\det(I)>0$ and $\det(w^{(1)}, \dots, w^{(n)})>0$.
Therefore, by the conditions (i) and (ii) of $\Phi$, we see that
\[
\int_{\Gamma_{w,+}\bs \Delta^{\circ}_{w,+}} \eta = \sum_{I \in \Phi^{(n)}} \int_{\Delta^{\circ}_I} \eta
\]
for any $\Gamma_{w,+}$-invariant $(n-1)$-form on $Y_{w,+}$. 
Hence we find
\begin{align}\label{eqn proof1}
  \begin{split}
&\int_{\Gamma_{w,+}\bs \Delta^{\circ}_{w,+}} 
N_{w^*}(y)^{k-1}\psi_{n(k-1), T_{w,+}}(y) \omega(y) \\
=&
\sum_{I \in \Phi^{(n)}} \int_{\Delta^{\circ}_I}
N_{w^*}(y)^{k-1}
\sum_{x \in T_{w,+}\cap \Z^n} \frac{1}{\brk{x, y}^{nk}}
\omega(y).
\end{split}
\end{align}

Now, regarding each $I \in \Phi^{(n)}$ as an element in $SL_n(\Z)$, we have
\[
\Delta^{\circ}_I=I\Delta^{\circ}_{(e_1, \dots, e_n)},
\]
where $e_1, \dots, e_n$ are the standard basis, i.e., $e_i=\tp(0, \dots, 0, \overset{i}{1}, 0, \dots, 0)$, and hence
\begin{align}\label{eqn proof2}
  \begin{split}
&
\sum_{I \in \Phi^{(n)}} \int_{\Delta^{\circ}_I}
N_{w^*}(y)^{k-1}
\sum_{x \in T_{w,+}\cap \Z^n} \frac{1}{\brk{x, y}^{nk}}
\omega(y) \\
=&
\sum_{I \in \Phi^{(n)}} \int_{\Delta^{\circ}_{(e_1, \dots, e_n)}}
N_{w^*}(I y)^{k-1}
\sum_{x \in T_{w,+}\cap \Z^n} \frac{1}{\brk{x, I y}^{nk}}
\omega(Iy) \\
=&
\sum_{I \in \Phi^{(n)}} \int_{\Delta^{\circ}_{(e_1, \dots, e_n)}}
N_{w^*}(I y)^{k-1}
\sum_{x \in \tp I T_{w,+}\cap \Z^n} \frac{1}{\brk{x, y}^{nk}}
\omega(y).
\end{split}
\end{align}
Now, we expand $N_{w^*}(I y)^{k-1}$ and set
\[
N_{w^*}(I y)^{k-1}=: \sum_{\substack{\bm k \in \Z_{\geq 0}^{n}\\ |\bm k|={n}(k-1)}} c_{I,\bmk} y^{\bm k} \in \Q[y_1, \dots, y_n]
\]
for some $c_{I,\bm k} \in \Q$, i.e., $c_{I,\bm k}$ is the coefficient of $y^{\bm k}=y_1^{k_1}\cdots y_n^{k_n}$ in $N_{w^*}(I y)^{k-1}$. 
Then by using Proposition \ref{prop feynman} again (with $\alpha_i=e_i$), we further find
\begin{align}\label{eqn proof3}
  \begin{split}
&
\sum_{I \in \Phi^{(n)}} \int_{\Delta^{\circ}_{(e_1, \dots, e_n)}}
N_{w^*}(I y)^{k-1}
\sum_{x \in \tp I T_{w,+}\cap \Z^n} \frac{1}{\brk{x, y}^{nk}}
\omega(y) \\
=&
\sum_{I \in \Phi^{(n)}}
\sum_{\substack{\bm k \in \Z_{ \geq 0}^{n}\\ |\bm k|={n}(k-1)}}
c_{I,\bm k}
\sum_{x \in \tp I T_{w,+}\cap \Z^n} 
\int_{\Delta^{\circ}_{(e_1, \dots, e_n)}}
y^{\bm k}
\frac{1}{\brk{x, y}^{nk}}
\omega(y) \\
=&
\sum_{I \in \Phi^{(n)}}
\sum_{\substack{\bm k \in \Z_{\geq 0}^{n}\\ |\bm k|={n}(k-1)}}
\frac{\bm k!}{(nk-1)!}
c_{I,\bm k}
\sum_{x \in \tp I T_{w,+}\cap \Z^n} 
\frac{1}{x^{\bm k + \mathbf{1}}} \\
=&
\frac{1}{(nk-1)!}
\sum_{I \in \Phi^{(n)}}
\sum_{\substack{\bm k \in \Z_{\geq 0}^{n}\\ |\bm k|={n}(k-1)}}
\bm k!
c_{I,\bm k}
\zeta_{\tp I T_{w,+}}(\bm k + \mathbf{1}),
\end{split}
\end{align}
where $\mathbf{1}=(1, \dots, 1)$ and $\bm k + \mathbf{1}=(k_1+1, \dots, k_n+1)$ for $\bm k=(k_1, \dots, k_n)$. 
Therefore, by combining Proposition \ref{prop int rep}, \eqref{eqn proof1}, \eqref{eqn proof2}, and \eqref{eqn proof3}, we obtain
\begin{align}\label{eqn main formula}
\zeta_{F,+}(\mfa^{-1}, k)
=
\frac{1}{((k-1)!)^n \sqrt{d_{F}} N\mfa}
\sum_{I \in \Phi^{(n)}}
\sum_{\substack{\bm k \in \Z_{\geq 1}^{n}\\ |\bm k|={n}k}}
(\bm k-\mathbf{1})!
c_{I,\bm k-\mathbf{1}}
\zeta_{\tp I T_{w,+}}(\bm k), 
\end{align}
where $\bm k - \mathbf{1}=(k_1-1, \dots, k_n-1)$. 
Finally, for $I \in \Phi^{(n)} \subset  (C_{(w^{(1)}, \dots, w^{(n)})})^n$, by Lemma \ref{lem basic} (\ref{item dual basis}), we see that $\tp I T_{w,+}$ is a totally positive cone, and by Lemma \ref{lem basic} (\ref{item dual cone}), we see that $\tp I T_{w,+}$ is also $F'$-rational. 
This completes the proof. 
\end{proof}

\section{Examples}\label{sec example}

In this section, we illustrate our main theorem with some examples.

\subsection{The case of $F=\Q(\sqrt{5})$}\label{sec ex sqrt5}

In this subsection, we consider the case where $n=2$, $F=\Q(\sqrt{5})$, and $\mfa = \mathcal O_F=\Z[\frac{1+\sqrt{5}}{2}]$, i.e., Example \ref{ex sqrt5}.

In this case, first we have
\begin{gather*}
\mathcal O_{F}^{\times} =\{\pm 1\} \times \left\{\left(\frac{1+\sqrt{5}}{2}\right)^{\nu} \ \Bigg|\  \nu \in \Z \right\}, \\ 
  \mathcal O_{F,+}^{\times} =\left\{\left(\frac{3+\sqrt{5}}{2}\right)^{\nu} \ \Bigg|\  \nu \in \Z \right\}
%
\end{gather*}
and hence
\begin{align*}
\zeta_{F,+}(\mfa^{-1}, s)
&= \sum_{x \in \mathcal O_{F,+}/\mathcal O_{F,+}^{\times}} \frac{1}{N_{F/\Q}(x)^s} \\
&= \sum_{x \in (\mathcal O_F - \{0\})/\mathcal O_F^{\times}} \frac{1}{|N_{F/\Q}(x)|^s} \\
&=\zeta_{\Q(\sqrt{5})}(s),
\end{align*}
where $\zeta_{\Q(\sqrt{5})}(s)$ is the Dedekind zeta function of $\Q(\sqrt{5})$.
Moreover, we have $d_F=5$. 

Now, let us choose $w:=\tp\left(\frac{1+\sqrt{5}}{2}, 1\right)$ as a basis of $\mfa$ over $\Z$. 
Then the dual basis $w^*=\tp(w_1^*,w_2^*)$, the dual norm polynomial $N_{w^*}$, and the $F$-rational cone $T_{w,+}$ can be computed as follows:
\begin{gather*}
  w^*=
  \tp\left(\frac{1}{\sqrt{5}}, \frac{-1+\sqrt{5}}{2\sqrt{5}}\right), \\
  N_{w^*}(x_1, x_2)=\frac{1}{5}(-x_1^2+x_1x_2+x_2^2), \\
  T_{w,+}=
  \Rpos
\begin{pmatrix}
\frac{1}{\sqrt{5}} \\
\frac{-1+\sqrt{5}}{2\sqrt{5}} \\
\end{pmatrix}
+ \Rpos
\begin{pmatrix}
-\frac{1}{\sqrt{5}} \\
\frac{1+\sqrt{5}}{2\sqrt{5}} \\
\end{pmatrix}
\\
\phantom{T_{w,+}}  =
\Rpos
\begin{pmatrix}
  \frac{1+\sqrt{5}}{2} \\
  1\\
\end{pmatrix}
+ \Rpos
\begin{pmatrix}
  \frac{1-\sqrt{5}}{2} \\
  1\\
\end{pmatrix}. 
\end{gather*}
Moreover, the cone decomposition $\Phi$ in the sense of Proposition \ref{prop shintani} can be taken as $\Phi=\{I, J\}$ with
\begin{align*}
  I=
  \left(
  \begin{pmatrix}
    1\\
    1\\
  \end{pmatrix},
  \begin{pmatrix}
    0\\
    1\\
  \end{pmatrix}
  \right),\quad
  J=
  \left(
  \begin{pmatrix}
    0\\
    1\\
  \end{pmatrix}
  \right).  
\end{align*}
Then its two-dimensional part $\Phi^{(2)}$ (cf. \eqref{eqn n-dim}) becomes $\Phi^{(2)}=\{I\}$. 

Therefore,
by putting
\begin{align*}
  C:= \tp I T_{w,+} =
  \Rpos
\begin{pmatrix}
  \frac{3+\sqrt{5}}{2} \\
  1\\
\end{pmatrix}
+ \Rpos
  \begin{pmatrix}
  \frac{3-\sqrt{5}}{2} \\
  1\\
\end{pmatrix}, 
\end{align*}
we obtain the following from \eqref{eqn main formula} in the proof of Theorem \ref{thm main}.
\begin{cor}
For $k \in \Z_{\geq 2}$, we have 
\begin{align*}
\zeta_{\Q(\sqrt{5})}(k)
=
\frac{1}{((k-1)!)^2 \sqrt{5}}
\sum_{\substack{(k_1, k_2) \in \Z_{\geq 1}^{2}\\ k_1+k_2=2k}}
(k_1-1)!(k_2-1)!
c_{I,(k_1-1, k_2-1)}
\zeta_{C}(k_1, k_2), 
\end{align*}
where $c_{I,(k_1-1, k_2-1)}$ is the coefficient of $y_1^{k_1-1}y_2^{k_2-1}$ in
\begin{align*}
  N_{w^*}(Iy)^{k-1}&=N_{w^*}(y_1,y_1+y_2)^{k-1}
  = \left(\frac{1}{5}(y_1^2+3y_1y_2+y_2^2)\right)^{k-1}. 
\end{align*}
\end{cor}
This shows \eqref{eqn zeta} in Example \ref{ex sqrt5}.
In the cases where $k=2,3$, the coefficients $c_{I,(k_1-1, k_2-1)}$ can be computed by
\begin{align*}
    N_{w^*}(Iy)^{2-1}&=\frac{1}{5}(y_1^2+3y_1y_2+y_2^2), \\
    N_{w^*}(Iy)^{3-1}&=\frac{1}{25}(y_1^4+6y_1^3y_2+11y_1^2y_2^2+6y_1y_2^3+y_2^4). 
\end{align*}
Moreover, we see that $\zeta_C(k_1,k_2)=\zeta_C(k_2, k_1)$ from a simple observation that $(x_1,x_2)\in C$ if and only if $(x_2,x_1)\in C$.
Thus we find 
\begin{align*}
\zeta_{\Q(\sqrt{5})}(2)
&=
\frac{1}{5\sqrt{5}}
\left(
2\zeta_C(3,1) + 3\zeta_C(2,2) +2\zeta_C(1,3) 
\right) \\
&=
\frac{1}{5\sqrt{5}}
\left(
4\zeta_C(3,1) + 3\zeta_C(2,2)
\right), \\
\zeta_{\Q(\sqrt{5})}(3)
&=
\frac{1}{4 \cdot 25\sqrt{5}}
(
24 \zeta_C(5,1) + 6 \cdot 6 \zeta_C(4,2) +4 \cdot 11\zeta_C(3,3) \\
&\quad\quad\quad\quad\quad\quad\quad +6 \cdot 6\zeta_C(2,4) +24 \zeta_C(1,5)
) \\
&=
\frac{1}{25\sqrt{5}}
\left(
12 \zeta_C(5,1) + 18 \zeta_C(4,2) + 11\zeta_C(3,3) 
\right), 
\end{align*}
which shows \eqref{eqn zeta2} and \eqref{eqn zeta3} in Example \ref{ex sqrt5}.

\subsection{The case of $F=\Q(\cos(\frac{2\pi}{7}))$}

Let
\begin{align*}
\eta=\eta^{(1)}:=2\cos(\frac{2\pi}{7}), \quad \eta^{(2)}:=2\cos(\frac{4\pi}{7}), \quad \eta^{(3)}:=2\cos(\frac{6\pi}{7})
\end{align*}
be the three roots of the cubic polynomial $X^3+X^2-2X-1$.
Note that we have $\eta^{(2)}=\eta^2-2$, $\eta^{(3)}=-\eta^2-\eta+1$. Thus $F:=\Q(\eta)$ is a totally real cubic field, and it is known that its ring of integers is $\Z[\eta]$ whose ideal class group is trivial and $d_F=49$.

In this subsection, we consider the case where $n=3$, $F=\Q(\eta)$, $\mfa=\mathcal O_F=\Z[\eta]$.
In this case, by setting
\begin{align*}
\varepsilon_1:= \eta^2-1, \quad \varepsilon_2:=\eta^2+\eta-2,
\end{align*}
it is known that
\begin{gather*}
  \mathcal O_F^{\times}=\{\pm 1\}\times \varepsilon_1^{\Z} \times \varepsilon_2^{\Z},\\
  \mathcal O_{F,+}^{\times}=\varepsilon_1^{2\Z} \times \varepsilon_2^{2\Z},  
\end{gather*}
and hence
\begin{align*}
\zeta_F(\mfa^{-1},s)=\zeta_{\Q(\eta)}(s), 
\end{align*}
where $\zeta_{\Q(\eta)}(s)$ is the Dedekind zeta function of $F=\Q(\eta)$.

Now, let us choose $w:=\tp(\eta^2,\eta,1)$ as a basis of $\mfa$ over $\Z$. Then the dual basis $w^*=(w_1^*, w_2^*, w_3^*)$, the dual norm polynomial $N_{w^*}$, and the $F'(=F)$-rational cone $T_{w,+}$ can be computed as follows:
\[
w^*=\tp \left(\frac{1}{7}(2\eta^2 + \eta - 3), \frac{1}{7}(\eta^2 + 2\eta - 1), \frac{1}{7}(-3\eta^2 - \eta + 7)\right),
\]
\begin{align*}
N_{w^*}(x_1,x_2,x_3)=
\frac{1}{49}
(&-x_1^3 -2x_1^2x_2 + 2x_1^2x_3 + x_1x_2^2 + 5x_1x_2x_3 \\
&+ x_1x_3^2 + x_2^3 - 3x_2^2x_3 - 4x_2x_3^2 - x_3^3),
\end{align*}
\begin{align*}
T_{w,+}=\sum_{i=1}^3\Rpos w^{*(i)},
\end{align*}
with
\[
w^{(i)}=\tp \left(\frac{1}{7}(2(\eta^{(i)})^2 + \eta^{(i)} - 3), \frac{1}{7}((\eta^{(i)})^2 + 2\eta^{(i)} - 1), \frac{1}{7}(-3(\eta^{(i)})^2 - \eta^{(i)} + 7)\right)
\]
for $i=1,2,3$.

Next, we describe the cone decomposition. 
Put 
\begin{align*}
  \alpha_0:=
  \begin{pmatrix}
    1\\
    0\\
    1\\
  \end{pmatrix}, 
  \alpha_1:=
  \begin{pmatrix}
    2\\
    -1\\
    1\\
  \end{pmatrix}, 
  \alpha_2:=
  \begin{pmatrix}
    1\\
    -1\\
    3\\
  \end{pmatrix}, 
  \alpha_3:=
  \begin{pmatrix}
    1\\
    -1\\
    2\\
  \end{pmatrix}, 
  \alpha_4:=
  \begin{pmatrix}
    2\\
    -1\\
    2\\
  \end{pmatrix}, 
\end{align*}
and set
\begin{align*}
&  I_1:=(\alpha_0, \alpha_2, \alpha_3),
  I_2:=(\alpha_0, \alpha_3, \alpha_4),
  I_3:=(\alpha_0, \alpha_4, \alpha_1),
  I_4:=(\alpha_4, \alpha_3, \alpha_1),\\
&  I_5:=(\alpha_0, \alpha_1),
  I_6:=(\alpha_0, \alpha_2),
  I_7:=(\alpha_0, \alpha_3),
  I_8:=(\alpha_0, \alpha_4),\\
&  I_9:=(\alpha_1, \alpha_4),
  I_{10}:=(\alpha_3, \alpha_4),
  I_{11}:=(\alpha_0),
  I_{12}:=(\alpha_1). 
\end{align*}
Note that $\alpha_0, \alpha_1, \alpha_2, \alpha_3$ are chosen so that $\alpha_1=\rho_{w^*}(\varepsilon_1^2)\alpha_0$, $\alpha_2=\rho_{w^*}(\varepsilon_2^2)\alpha_0$, $\alpha_3=\rho_{w^*}(\varepsilon_1^2\varepsilon_2^2)\alpha_0$, and $\alpha_4$ is an auxiliary vector to make cones smooth. 
Then, by using \cite[Lemme 2.2]{colmez} with totally positive fundamental units $\varepsilon_1^2, \varepsilon_2^2$, we find that $\Phi=\{I_i \mid i=1, \dots, 12\}$ gives a cone decomposition in the sense of Proposition \ref{prop shintani}. 
In this case, the three-dimensional part $\Phi^{(3)}$ becomes
\[
\Phi^{(3)}=\{I_1, I_2, I_3, I_4\}. 
\]

Therefore, by setting
\begin{align*}
C_i:= \tp I_i T_{w,+}, \quad i=1,\dots, 4,
\end{align*}
we obtain the following form \eqref{eqn main formula}.

\begin{cor}
For $k \in \Z_{\geq 2}$, we have 
\begin{align*}
\zeta_{\Q(\eta)}(k)
=
\frac{1}{7 ((k-1)!)^3}
\sum_{i=1}^4
\sum_{\substack{\bm k \in \Z_{\geq 1}^{3}\\ |\bm k|=3k}}
(\bm k - \mathbf{1})!
c_{I_i,\bm k - \mathbf{1}}
\zeta_{C_i}(\bm k), 
\end{align*}
where $c_{I_i,\bm k - \mathbf{1}}$ is the coefficient of $y_1^{k_1-1}y_2^{k_2-1}y_3^{k_3-1}$ in $N_{w^*}(I_iy)^{k-1}$ for $\bm k=(k_1, k_2, k_3)$.
\end{cor}

In the cases where $k=2, 3$, by computing the coefficients $c_{I_i, \bm k -\mathbf{1}}$ explicitly, we find the following:
\begin{align}
\zeta_{\Q(\eta)}(2)
=
\frac{1}{7^3}
\sum_{i=1}^4
\sum_{\substack{\bm k \in \Z_{\geq 1}^{3}\\ |\bm k|=6}}
c'_{i,\bm k - \mathbf{1}}
\zeta_{C_i}(\bm k),\\
\zeta_{\Q(\eta)}(3)
=
\frac{1}{7^5}
\sum_{i=1}^4
\sum_{\substack{\bm k \in \Z_{\geq 1}^{3}\\ |\bm k|=9}}
c'_{i,\bm k - \mathbf{1}}
\zeta_{C_i}(\bm k),
\end{align}
where the coefficients $c'_{i,\bm k - \mathbf{1}}$ are given in Table \ref{table zeta2} and Table \ref{table zeta3}.

\begin{table}[h]
\begin{tabular}{c| cccc l |c| cccc}
$\bm k$ & \quad$c'_{1, \bm k}$\quad & \quad$c'_{2, \bm k}$\quad & \quad$c'_{3, \bm k}$\quad & \quad$c'_{4, \bm k}$\quad &  & $\bm k$ & \quad$c'_{1, \bm k}$\quad & \quad$c'_{2, \bm k}$\quad & \quad$c'_{3, \bm k}$\quad & \quad$c'_{4, \bm k}$\quad \\ \hline
(4,1,1) & 6      & 6      & 6      & 42     &  & (2,1,3) & 12     & 28     & 10     & 14     \\
(3,2,1) & 12     & 10     & 14     & 28     &  & (1,4,1) & 6      & 6      & 42     & 6      \\
(3,1,2) & 10     & 14     & 12     & 28     &  & (1,3,2) & 12     & 14     & 28     & 10     \\
(2,3,1) & 10     & 12     & 28     & 14     &  & (1,2,3) & 10     & 28     & 14     & 12     \\
(2,2,2) & 13     & 21     & 21     & 21     &  & (1,1,4) & 6      & 42     & 6      & 6     
\end{tabular}
\caption{$c'_{i, \bm k}$ for $k=2$}\label{table zeta2}
\end{table}
\begin{table}[h]
\begin{tabular}{c| cccc l |c| cccc}
$\bm k$ & \quad$c'_{1, \bm k}$\quad & \quad$c'_{2, \bm k}$\quad & \quad$c'_{3, \bm k}$\quad & \quad$c'_{4, \bm k}$\quad &  & $\bm k$ & \quad$c'_{1, \bm k}$\quad & \quad$c'_{2, \bm k}$\quad & \quad$c'_{3, \bm k}$\quad & \quad$c'_{4, \bm k}$\quad \\ \hline
(7,1,1) & 90     & 90     & 90     & 4410   &  & (3,1,5) & 276    & 1764   & 222    & 462    \\
(6,2,1) & 180    & 150    & 210    & 2940   &  & (2,6,1) & 150    & 180    & 2940   & 210    \\
(6,1,2) & 150    & 210    & 180    & 2940   &  & (2,5,2) & 258    & 378    & 2058   & 336    \\
(5,3,1) & 276    & 222    & 462    & 1764   &  & (2,4,3) & 327    & 735    & 1281   & 462    \\
(5,2,2) & 258    & 336    & 378    & 2058   &  & (2,3,4) & 318    & 1302   & 693    & 504    \\
(5,1,3) & 222    & 462    & 276    & 1764   &  & (2,2,5) & 258    & 2058   & 336    & 378    \\
(4,4,1) & 279    & 279    & 945    & 945    &  & (2,1,6) & 180    & 2940   & 150    & 210    \\
(4,3,2) & 327    & 462    & 735    & 1281   &  & (1,7,1) & 90     & 90     & 4410   & 90     \\
(4,2,3) & 318    & 693    & 504    & 1302   &  & (1,6,2) & 180    & 210    & 2940   & 150    \\
(4,1,4) & 279    & 945    & 279    & 945    &  & (1,5,3) & 276    & 462    & 1764   & 222    \\
(3,5,1) & 222    & 276    & 1764   & 462    &  & (1,4,4) & 279    & 945    & 945    & 279    \\
(3,4,2) & 318    & 504    & 1302   & 693    &  & (1,3,5) & 222    & 1764   & 462    & 276    \\
(3,3,3) & 349    & 847    & 847    & 847    &  & (1,2,6) & 150    & 2940   & 210    & 180    \\
(3,2,4) & 327    & 1281   & 462    & 735    &  & (1,1,7) & 90     & 4410   & 90     & 90    
\end{tabular}
\caption{$c'_{i, \bm k}$ for $k=3$}\label{table zeta3}
\end{table}



\begin{thebibliography}{99}
\bibitem{amrt}
A. Ash\ et al., {\it Smooth compactifications of locally symmetric varieties}, second edition, Cambridge Mathematical Library, Cambridge University Press, Cambridge, 2010. MR2590897



\bibitem{bekki}
H.~Bekki, Shintani-Barnes cocycles and values of the zeta functions of
  algebraic number fields, to appear in Algebra \& Number Theory, arXiv:2104.09030. 
%
%
%


\bibitem{colmez}
P. Colmez, R\'{e}sidu en $s=1$ des fonctions z\^{e}ta $p$-adiques, Invent. Math. {\bf 91} (1988), no.~2, 371--389. MR0922806

  
%
\bibitem{cls}
D. A. Cox, J. B. Little\ and\ H. K. Schenck, {\it Toric varieties}, Graduate Studies in Mathematics, 124, American Mathematical Society, Providence, RI, 2011. MR2810322

%
\bibitem{gpz}
L. Guo, S. Paycha\ and\ B. Zhang, Conical zeta values and their double subdivision relations, Adv. Math. {\bf 252} (2014), 343--381. MR3144233
%
%
%
%
%
\bibitem{hurwitz}
A.~Hurwitz.
\newblock \"{U}ber die {A}nzahl der {K}lassen positiver tern\"{a}rer
  quadratischer {F}ormen von gegebener {D}eterminante.
\newblock {\em Math. Ann.}, 88(1-2):26--52, 1922.
%
%
%
%

\bibitem{ishida}
M.-N. Ishida, The duality of cusp singularities, Math. Ann. {\bf 294} (1992), no.~1, 81--97. MR1180451
%
%
%
%
%
\bibitem{shintani}
T.~Shintani.
\newblock On evaluation of zeta functions of totally real algebraic number
  fields at non-positive integers.
\newblock {\em J. Fac. Sci. Univ. Tokyo Sect. IA Math.}, 23(2):393--417, 1976.
%
%
%
%

\bibitem{terasoma}
T.~Terasoma,  
Rational convex cones and cyclotomic multiple zeta values, 
arXiv:math/0410306. 
%
\end{thebibliography}
\end{document}